\def\mytitle{A linear time algorithm to verify strong structural controllability}
\def\confidentialstring{%
Submitted to ArXiv
on \today.}
\let\headnote=\relax
\def\mykeywords{Strong structural controllability, linear time algorithm, sparse matrix}

\documentclass[letterpaper, 10 pt, conference]{myieeeconf}  

\IEEEoverridecommandlockouts                              

\usepackage{graphicx}
\usepackage{cite}
\usepackage{psfrag}
\usepackage{xcolor}
\usepackage{algorithmic}
\usepackage{algorithm}
\usepackage{ifthen}
\makeatletter
\let\proof\@undefined
\let\endproof\@undefined
\makeatother

\usepackage{gunther2e}                   

\newboolean{Forreview}
\setboolean{Forreview}{true}

\title{\textbf{%
\mytitle\thanks{This is the accepted version of a paper published in \textit{Proc. 53rd IEEE Conf. Decision and Control}, 2014.}%
}}

\author{Alexander Weber, Gunther Reissig and Ferdinand Svaricek%
\thanks{This work has been supported by the German Research Foundation (DFG) under grant no. RE 1249/3-2. The authors are with the University of the Federal Armed Forces Munich,
Dept. Aerospace Eng., Inst. of Control Eng. (LRT-15), D-85577 Neubiberg
(Munich), Germany.}}%

\ifx
\hypersetup\undefined\def\href#1#2{\texttt{#2}}
\else
\hypersetup{
pdftitle={\mytitle},
pdfauthor={Alexander Weber, Gunther Reissig, Ferdinand Svaricek},%
pdfsubject={\confidentialstring{} \headnote},
pdfkeywords={\mykeywords}
}
\fi

\ifthenelse{\boolean{Forreview}}{}{\edef\baselinestretch{.975}}

\newcommand{\sA}{\mathcal{A}}
\newcommand{\sB}{\mathcal{B}}
\newcommand{\sX}{\mathcal{X}}
\newcommand{\nz}{\ast}

\newcommand{\myvspace}{\vspace{2mm}}

\begin{document}

\maketitle
\thispagestyle{empty}
\pagestyle{empty}

\makeatletter
\def\@endtheorem{\endtrivlist\unskip}
\def\endproof{\hspace*{\fill}~\popQED\par\endtrivlist\unskip}
\def\IEEEproofname{Proof}
\def\@IEEEproof[#1]{\par\noindent\hspace{0em}{\itshape #1: }}
\def\proof{\pushQED{\qed}\@ifnextchar[{\@IEEEproof}{\@IEEEproof[\IEEEproofname]}}
\def\endproof{\hspace*{\fill}~\popQED\par\endtrivlist\unskip}
\makeatother

\begin{abstract}
We prove that strong structural controllability of a pair of
structural matrices $(\sA,\sB)$ can be verified in time linear in 
$n + r + \nu$, where $\sA$ is square, $n$ and $r$ denote the number of columns of $\sA$
and $\sB$, respectively, and $\nu$ is the number of non-zero entries
in $(\sA,\sB)$. We also present an algorithm realizing this bound,
which depends on a recent, high-level method to verify strong
structural controllability and uses sparse matrix data
structures. Linear time complexity is actually achieved by
separately storing both the structural matrix $(\sA,\sB)$ and its
transpose, linking the two data structures through a third one, and
a novel, efficient scheme to update all the data during the
computations. We illustrate the performance of our algorithm using
systems of various sizes and sparsity.
\end{abstract}

\section{Introduction}
\label{s:intro}
Strong structural controllability of the pair $(\sA,\sB)$ of structural matrices $\sA \in \{0,\nz\}^{n \times n}$, $\sB \in \{0 , \nz\}^{n \times r}$ is, by definition, equivalent to the linear system 
\begin{equation}
\label{e:system}
\dot x(t) = Ax(t) + Bu(t)
\end{equation}
being controllable for \emph{all} matrices $A$ and $B$ whose positions
of the non-zero entries (zero entries) coincide with the positions of
the \begriff{$\nz$-entries} (\begriff{$0$-entries}) of $\sA$ and $\sB$,
respectively. Here, $A$ and $B$ denote matrices with real or complex
entries having the same dimension as $\sA$ and $\sB$, respectively,
$x$ denotes the real or complex valued $n$-dimensional state of
\ref{e:system} and $u$ is a real or complex valued $r$-dimensional
input signal. The system given by \ref{e:system} is controllable if
for any initial state and any terminal state, there exists an input
signal $u$ steering the system from the initial to the terminal state
\cite{Sontag98}.

Strong structural controllability of linear time-invariant systems has
been extensively studied
\cite{MayedaYamada79,ReinschkeSvaricekWend92,JarczykSvaricekAlt11,ChapmanMesbahi13}. Algorithms to test strong structural controllability of a pair $(\sA,\sB)$ have been presented in \cite{ReinschkeSvaricekWend92} and \cite{ChapmanMesbahi13} having complexity $\mathcal{O}(n^3)$ and $\mathcal{O}(n^2)$, respectively. In \cite{i12str}, an algorithm was presented without an analysis of its complexity. 

Recently,
the notion of strong structural controllability has been extended to linear time-varying systems and
characterizations in terms of the zero-nonzero pattern $(\sA,\sB)$ have been established
\cite{i12str,i13str,i13strb,i14str}. While the conditions differ, it turns out that their
verification for a time-varying system can be reduced to the verification of strong structural controllability for an auxiliary time-invariant system \ref{e:system}. This implies that algorithms originally derived to test the strong structural controllability of time-invariant systems may be also used for the time-varying case. 

In this paper, we prove that strong structural controllability can be
verified in time linear in  $n + r + \nu$, where $\nu$ is the number
of non-zero entries in $(\sA,\sB)$.
We also present an algorithm realizing this bound,
which depends on the recent, high-level method from
\cite{i12str,i14str} and uses sparse matrix data
structures. Linear time complexity is actually achieved by
separately storing both the structural matrix $(\sA,\sB)$ and its
transpose, linking the two data structures through a third one, and
a novel, efficient scheme to update all the data during the computations.

The need for fast algorithms becomes evident by the following
application of strong structural controllability. The dynamical
evolution of complex networks, such as power grids or gene regulatory
networks, is commonly studied in terms of linear systems of the form
\ref{e:system}, where the entries of $x$ denote the state of the
nodes, $A$ denotes the adjacency matrix of the underlying graph and
$B$ identifies the nodes that can be controlled from outside the
network; see e.g. \cite{LiuSlotineBarabasi11} and the references
therein. In real applications,
the entries of the matrix $A$ are not exactly known, which is why one
considers its zero-nonzero structure, encoded in the structural
matrices $\sA$ and $\sB$, instead.
The particular interest with regard to controllability of networks is
then to find a structural matrix $\sB$ with the minimum number of
columns such that the given network is strong structurally
controllable \cite{ChapmanMesbahi13,PequitoPopliKarIlicAguiar13}. This
problem was proved to be NP-hard \cite{ChapmanMesbahi13}. One way to
avoid NP-hardness is to consider the special case in which $\sB$ is
required to have precisely one $\nz$-entry
per column, which results in $\mathcal{O}(n^3)$ time-complexity
\cite{PequitoPopliKarIlicAguiar13}. Another alternative is to pose the
problem in the framework of the so-called weak structural
controllability
\cite{LiuSlotineBarabasi11,LiuSlotineBarabasi12,CommaultDion13}. However,
both alternatives suffer from severe drawbacks. Firstly, restricting
$\sB$ to some special structure may result in a minimum number of
columns that is strictly greater than the number of columns actually
required using arbitrary $\sB$. (An example is given in the present
paper.) With regard to economizing the computational effort for input
signals that solution is inappropriate. Secondly, the approach based
on weak structural controllability yields results that are correct 
for
all pairs of matrices $(A,B)$ of structure
$(\sA,\sB)$ with the possible exception of a
set of measure zero. The possible exceptions may very well be a
problem, in particular, when the parameters of the system
\ref{e:system} slowly change over time, so that the submanifold of
exceptional points may be passed over with certainty.
Therefore, there is much interest to tackle the original NP-hard problem based
on strong rather than weak controllability, and fast algorithms are in
demand.

The remainder of this paper is organized as follows. 
Having introduced some notation and terminology in Section
\ref{s:notation}, in Section \ref{s:Review} we briefly review the
method presented in \cite{i12str,i14str}.
Section \ref{s:implementation} contains the main result about the time complexity for verifying strong structural controllability and an implementable algorithm of such a test. In Section \ref{s:results}, several computational results on the performance of an implementation on various structural matrices are presented.

\section{Notation and terminology}
\label{s:notation}
The set $\{1,2,3,\ldots \}$ of natural numbers we denote by $\mathbb{N}$, the set of real and complex numbers by $\mathbb{R}$ and $\mathbb{C}$, respectively, and $\mathbb{F}$ denotes either $\mathbb{R}$ and $\mathbb{C}$. For $a,b \in \mathbb{N}$, $a\leq b$, we write $\intcc{a;b}$ and $\intco{a;b}$ for the set $\{a,a+1,\ldots,b\}$ and $\{a,a+1,\ldots,b-1\}$, respectively. For the $i$-th entry of $y \in \mathbb{N}^m$ we write $y(i)$ ($1\leq i\leq m$). Moreover, we write $y \in \intcc{a;b}^m$ if $y(i) \in \intcc{a;b}$ for all $i$.

$\sX$ stands for a \begriff{structural matrix}, i.e. $\sX \in \{0,\nz\}^{n\times m}$. We say that a matrix $X \in \mathbb{F}^{n \times m}$ \begriff{has the non-zero structure of} $\sX$ if $X_{i,j} \neq 0$ is equivalent to $\sX_{i,j} = \nz$ for any $i,j$. Here and subsequently, $X_{i,j}$ ($\sX_{i,j}$, respectively) denotes the entry in the $i$-th row and $j$-th column of $X$ ($\sX$, respectively). $\sA$ and $\sB$ denote structural matrices of dimension $n \times n$ and $n \times r$, respectively. 
The transpose of $\sX$ is denoted by $\sX^T$. For a structural matrix $\sX$ we introduce the following sets. For $j \in \intcc{1;m}$ we define $$\operatorname{NZR}_{\sX}(j)\defas \{ i \in \intcc{1;n} \ | \ \sX_{i,j} = \nz \}.$$ The above set indicates the rows of $\sX$ that have a $\nz$-entry in the $j$-th column. For reviewing the results in \cite{i12str} as outlined in the introduction, we define for a set $V \subseteq \intcc{1;n}$ the set $$\operatorname{NZC}_{\sX}(V)\defas \{ j \in \intcc{1;m} \ | \ \exists_{i \in V} : \ \mathcal{X}_{i,j} = \nz \}.$$ 
Throughout the paper, however, we will omit the subscript $\sX$ as it will be obvious from the context to which matrix the sets are related. 

\section{Review of the method to be implemented}
\label{s:Review}
In this section, we state the method for testing strong structural controllability as given in \cite{i12str} for which we will give an implementable algorithm in the subsequent section. The test consists of computing the set $V$ as specified in \ref{fig:alg:1} for both $L = 0$ and $L=1$. (We adopted the formulation of the test as presented in \cite{i14str}.) For convenience of the reader, we will briefly indicate the role of the two runs by stating the theorem which implies the correctness of the method. 
\begin{definition}
The pair $(\sA,\sB)$ of structural matrices is \begriff{strong structurally controllable for $\lambda \in \mathbb{C}$} if the matrix $(\lambda \operatorname{id} - A, B)$ has full rank for all pairs of matrices $(A,B) \in \mathbb{F}^{n \times (n+r)}$ that have the non-zero structure of $(\sA,\sB)$. Here, $\operatorname{id}$ denotes the $n\times n$ identity matrix.
\end{definition}
A consequence of the well-known Hautus criterion (e.g.\linebreak \cite[Lemma 3.3.7]{Sontag98}) is that the pair $(\sA,\sB)$ is strong structurally controllable if and only if it is strong structurally controllable for all $\lambda \in \mathbb{C}$. Based on this fact, the following theorem has been proved in \cite{MayedaYamada79}.

\begin{theorem}
Consider the following conditions for the structural matrix $\sX = (\sA,\sB)$:
\begin{asparaenum}
\item[(${G}_0$)]{ For every non-empty subset $V \subseteq \intcc{1;n}$ of row indices of $\sX$ there exists a column index $v \in \intcc{1;n+r}$ such that $V \cap \operatorname{NZR}(v)$ is a singleton,
\label{th:maincitedthm:condition:0}
}
\item[(${G}_1$)]{ For every non-empty subset $V \subseteq \intcc{1;n}$ of row indices of $\sX$ that satisfies $V \subseteq \operatorname{NZC}(V)$ there exists $v \in \intcc{1;n+r} \setminus V$ such that $V \cap \operatorname{NZR}(v)$ is a singleton.
\label{th:maincitedthm:condition:1}
}
\end{asparaenum}
Condition ($G_0$) holds if and only if $\sX$ is strong structurally controllable for $\lambda = 0$. Analogously, condition ($G_1$) holds if and only if $\sX$ is strong structurally controllable for every $\lambda \in \mathbb{C} \setminus \{0\}$. In particular, $\sX$ is strong structurally controllable if and only if both ($G_0$) and ($G_1$) hold.
\end{theorem}

The proof of the above theorem as given in \cite{i12str} proves the following theorem.
\begin{theorem}
\label{th:correctness}
The pair $(\sA,\sB)$ is strong structurally controllable 
\begin{asparaenum}[(i)]
\item for $\lambda = 0$ if and only if the algorithm in \ref{fig:alg:1} returns the empty set for $L=0$,
\item for every $\lambda \neq 0$ if and only if the algorithm in \ref{fig:alg:1} returns the empty set for $L=1$.
\end{asparaenum}
In particular, $(\sA,\sB)$ is strong structurally controllable if and only if both runs of the algorithm return the empty set.
\end{theorem}
\myvspace
It is important to note that although conditions ($G_0$) and ($G_1$) require verifications for every non-empty subset $V \subseteq \intcc{1;n}$, Theorem \ref{th:correctness} implies that a test of merely $n$ such subsets is sufficient. Nevertheless, a brute-force implementation of \ref{fig:alg:1} will not lead to a linear time test since the computation of the sets $T$ and $\operatorname{NZC}(V)$ is complex. 

In the following section, we present an algorithm that realizes the method given in \ref{fig:alg:1} in linear time. The key to linear time complexity is combining sophisticated data structures and sparse matrix techniques. 

\begin{figure}
\ifthenelse{\boolean{Forreview}}{\normalsize}{\small}
\algorithmicindent.6em%
\begin{algorithmic}[1]
\item[\textbf{Input}] $L$, $(\sA,\sB)$
\REQUIRE $L \in \{0,1\}$
\STATE $V \defas \intcc{1;n}$
\WHILE {$V \neq \emptyset$}
\IF {$L=0$}
\STATE
$T \defas \{ v \in \intcc{1;n+r} \ | \ |V \cap \operatorname{NZR}(v) | = 1 \}$
\label{alg1:updateT:0}
\ELSE
\STATE
$T \defas \{ v \in \intcc{1;n+r} \setminus V \ | \ |V \cap \operatorname{NZR}(v) | = 1 \}$
\label{alg1:updateT:1}
\ENDIF
\IF {$L=0$ \OR $V \subseteq \operatorname{NZC}(V)$}
\label{alg1:T0empty}
\IF {$T = \emptyset$}
\STATE \textbf{break}
\ENDIF
\STATE
Pick $v \in T$.
\STATE
$\{w\}\defas \operatorname{NZR}(v)$
\label{alg1:removefromV:1}
\ELSE 
\STATE
Pick $w \in V \setminus \operatorname{NZC}(V)$.
\label{alg1:removefromV:2}
\ENDIF
\label{alg1:line:16}
\STATE $V\defas V \setminus \{w\}$
\label{alg1:removefromV}
\ENDWHILE
\item[\textbf{Output}] $V$
\end{algorithmic}
\caption{\label{fig:alg:1}Method to test if $(\sA,\sB)$ is strong structurally controllable \cite{i12str}.}
 \vspace*{-\baselineskip}
\end{figure}
\normalsize
\begin{figure*}[t]
\centering
\psfrag{index}[][]{\small index}
\psfrag{s}[][]{\small$s$}
\psfrag{tildes}[][]{\small$\tilde s$}
\psfrag{is}[][]{\small$i_s$}
\psfrag{z}[][]{\small$z$}
\psfrag{tildez}[][]{\small$\tilde z$}
\psfrag{iz}[][]{\small$i_z$}
\psfrag{1}[][]{\small$1$}
\psfrag{2}[][]{\small$2$}
\psfrag{3}[][]{\small$3$}
\psfrag{4}[][]{\small$4$}
\psfrag{5}[][]{\small$5$}
\psfrag{6}[][]{\small$6$}
\psfrag{7}[][]{\small$7$}
\psfrag{8}[][]{\small$8$}
\psfrag{9}[][]{\small$9$}
\psfrag{10}[][]{\small$10$}
\psfrag{r1}[][]{\small row $1$}
\psfrag{r2}[][]{\small row $2$}
\psfrag{r3}[][]{\small row $3$}
\psfrag{r4}[][]{\small row\,$4$}
\psfrag{r5}[][]{\small row\,$5$}
\psfrag{r6}[][]{\small row$6$}
\psfrag{s1}[l][]{\small column $1$}
\psfrag{s2}[l][]{\small column $2$}
\psfrag{s3}[l][]{\small column $3$}
\psfrag{s4}[l][]{\small column $4$}
\psfrag{s5}[l][]{\small column $5$}
\psfrag{s6}[l][]{\small column $6$}
\psfrag{s7}[l][]{\small column $7$}
\psfrag{s8}[l][]{\small column $8$}
\includegraphics[scale=0.6]{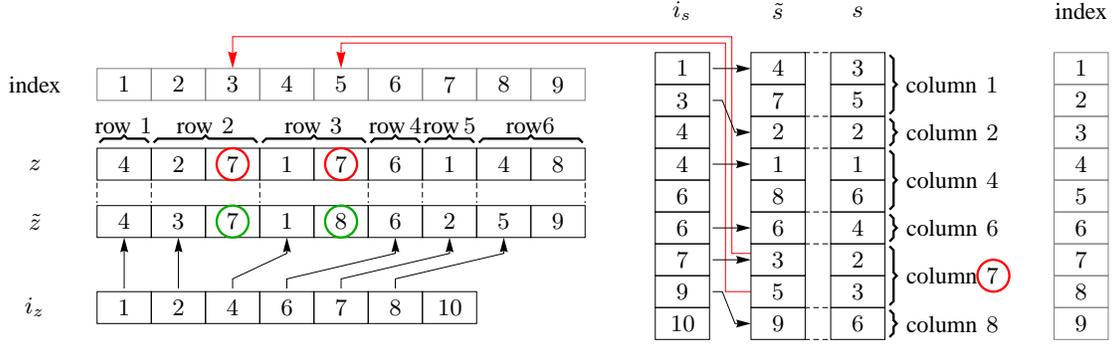}
\caption{\label{t:1} Data structures for $(\sA,\sB)$ in Example \ref{ex:ccs}. The usage of the array $\tilde s$ is indicated: The positions in $z$ whose entry is column $7$ are stored in positions $7$ and $8$ of $\tilde{s}$. $7$ and $8$ are the indices of column $7$. The entries of $\tilde z$ indicated by the green circles are those that need to be swapped in Example \ref{ex:tildes}.}
\vspace*{-\baselineskip}
\end{figure*}

\section{The main result}
\label{s:implementation}
Our main result, which claims the existence of a linear time test for strong structural controllability, is given in Section \ref{ss:main}. For its proof we give a particular algorithm for \ref{fig:alg:1}. Specifically, in Section \ref{ss:datastructures} we discuss the used data structures and the algorithm is presented in Section \ref{ss:operations}.
\subsection{Main result}
\label{ss:main}
\begin{theorem}
\label{th:main}
Let the pair of structural matrices \linebreak$(\sA,\sB) \in \{0,\nz\}^{n \times (n+r)}$ have $\nu \in \intcc{0;n(n+r)}$ $\nz$-entries. Strong structural controllability of $(\sA,\sB)$ can be verified with time complexity $\mathcal{O}(n+r+\nu)$. 
\end{theorem}

The following Sections \ref{ss:datastructures} and \ref{ss:operations} are devoted to the proof of Theorem \ref{th:main}.
Let us abbreviate the pair of structural matrices $(\sA,\sB)$ by $\sX$, and let $n$, $r$ and $\nu$ be as in the statement of Theorem \ref{th:main}. Without loss of generality, let $\nu >0$. 
\subsection{Data structures}
\label{ss:datastructures}
To obtain linear complexity in the algorithm that we present, we introduce the following sophisticated data structures. To begin with, we will store the matrices $\sX$ and $\sX^T$ \emph{separately}. The format that we use is well-known in the framework of sparse matrices \cite{BarrettEtAl94,i06Fill,i07Fill}. To provide efficient access between the data, we introduce a third, novel data structure which links those of $\sX$ and $\sX^T$. Moreover, we introduce appropriate data structures for the sets $T$ and $V$ as defined in \ref{fig:alg:1}. 
\subsubsection{Data structure for $\sX$}
The structural matrix $\sX$ is assumed to be available in the compressed column storage format (CCS) \cite{BarrettEtAl94}. The CCS-format exists in two versions, namely for ordinary matrices and for structural matrices. The latter, suitable for our purposes, consists of two integer arrays $s$ and $i_s$ of length $\nu$ and $n+r+1$, respectively. 
These arrays are defined as follows:
\begin{compactitem}[$\cdot$]
\item $i_s(j)-1$ equals the number $\nz$-entries in the first $j-1$ columns of $\sX$, and
\item  $\mathcal{X}_{i,j} = \nz$ if and only if there exists $k\in \intco{i_s(j);i_s(j+1)}$ such that $s(k) = i$.
\end{compactitem}
Note that $s \in \intcc{1;n}^\nu$ and $i_s \in \intcc{1;\nu+1}^{n+r+1}$. (See Section \ref{s:notation} for notation.) %
\begin{example}
\label{ex:ccs}
Consider the structural matrices 

\small
\begin{equation}
\label{e:ex:ccs}
\sA =\begin{pmatrix}
 0 & 0 & 0 & \nz & 0 & 0 \\
 0 & \nz & 0 & 0 & 0 & 0 \\
 \nz & 0 & 0 & 0 & 0 & 0 \\
 0 & 0 & 0 & 0 & 0 & \nz \\
 \nz & 0 & 0 & 0 & 0 & 0 \\
 0 & 0 & 0 & \nz & 0 & 0
\end{pmatrix}
\quad \text{and} \quad 
\sB = \begin{pmatrix}
0 & 0 \\
 \nz & 0 \\
 \nz & 0 \\
 0 & 0 \\
 0 & 0\\
 0 & \nz 
\end{pmatrix}. 
\end{equation}
\normalsize
The $\nz$-entries in column $1$ of $(\sA,\sB)$ are in the rows $3$ and $5$, hence $(s(1),s(2))$ may equal $(3,5)$ or $(5,3)$. We emphasize that both choices are consistent with our definition. $s(3)$ equals $2$ since the $\nz$-entry in column $2$ appears in row $2$. For the array $i_s$ we have $i_s(1) = 1$ as the first row index related to column $1$ is stored in position $1$ of $s$. $i_s(2) = 3$ since the first row index related to column $2$ is stored in position $3$ of $s$. The subsequent entries of $s$ and $i_s$ are obtained similarly. See also \ref{t:1}. 
\end{example}
\myvspace

\subsubsection{Data structure for $\sX^T$}
We store $\sX^T$ in its CCS-format and we denote the corresponding arrays by $z$ and $i_z$. Note that $z \in \intcc{1;n+r}^{\nu}$ and $i_z \in \intcc{1;\nu+1}^{n+1}$. This data structure is also known as the compressed row storage format of $\sX$ \cite{BarrettEtAl94}.

For simplicity of notation we introduce the following definition.
\begin{definition}
\label{d:index}
Let $i \in \intcc{1;n}$ and $j \in \intcc{1;n+r}$. We say that \begriff{$k \in \intcc{1;\nu}$ is an index of the row} $i$ if $k \in \intco{i_z(i);i_z(i+1)}$, and that \begriff{$l\in \intcc{1;\nu}$ is an index of the column $j$} if $l \in \intco{i_s(j);i_s(j+1)}$.
\end{definition}
\myvspace
\subsubsection{Data structures for linking the data structures of $\sX$ and $\sX^T$}
In the algorithm that we present, we will take advantage of the non-uniqueness of the array $s$ as follows. Entries in $s$ will be swapped during the execution of the algorithm in order to store additional information in the ordering of the entries of $s$ (without violating the properties of $s$ as a part of the CCS-format of $\sX$). The input of a swapping operation will be a row index $w$, and the first step will be to identify in constant time the positions $l$ in $s$ such that $s(l) =w$. The arrays $z$, $i_z$ provide the column indices $j$ such that $\sX_{w,j} = \nz$, which is the set $\{z(l) \ | \ l \in \intco{i_z(w);i_z(w+1)} \}$. However, $z$ and $i_z$ do not provide the \emph{positions} in $s$ of the entry $w$. 

In order to avoid a search operation, we introduce an integer array $\tilde z$ of length $\nu$ as follows. We define $\tilde z$ such that $\tilde{z}(l)$ equals the position in $s$ in which $w$ is stored among the row indices of the column $j=z(l)$.

Analogously, we will introduce an array $\tilde{s}$ to store the positions of the column indices in the array $z$. The array $\tilde{s}$ will be required to update $\tilde{z}$ as a swap in $s$ will require an update of $\tilde{z}$.

Before we define $\tilde{s}$ and $\tilde{z}$ formally, we identify some entries of $\tilde{z}$ for the pair $(\sA,\sB)$ as given in \ref{e:ex:ccs}.

\begin{example}
We consider the arrays $s$ and $z$ as given in \ref{t:1}. Row $2$ of $(\sA,\sB)$ has $\nz$-entries in columns $2=z(2)$ and $7=z(3)$. Among the row indices of column $2$ in $s$, row $2$ appears in position $3$, hence we set $\tilde{z}(2) = 3$. As for column $7$, row $2$ appears in position $7$ in $s$, hence $\tilde{z}(3) \defas 7$. 

Let us suppose that we had defined $z(2) = 7$ and $z(3) = 2$, so that the pair $(z,i_z)$ would still represent the pair $(\sA,\sB)$ given in \ref{e:ex:ccs}. In this case, we need to define $\tilde{z} (2)$ to equal $7$ since we require $\tilde{z}(2)$ to be an index pointing to row indices of column $z(2)$. Similarly, in this case, $\tilde{z}(3)=2$ since $2$ is an index of column $z(3)$. 
\end{example}

The integer arrays $\tilde s \in \intcc{1;\nu}^\nu$, $\tilde z \in \intcc{1;\nu}^\nu$ are formally defined by the following properties:
\begin{subequations}
\label{e:defoftildes}
\begin{align}
& \tilde s(k) \text{ is an index of the row }s(k) \text{ for all $k$, and } \label{e:tildes:a} \\
& z(\tilde s(l)) = j \text{ if $l$ is an index of the column $j$,} \label{e:tildes:b}
\end{align}
\end{subequations}
and similarly,
\begin{subequations}
\label{e:defoftildez}
\begin{align}
& \tilde z(l) \text{ is an index of the column }z(l)\text{ for all $l$, and }  \label{e:tildez:a} \\
& s(\tilde z(k)) = i \text{ if $k$ is an index of the row $i$}. \label{e:tildez:b}
\end{align}
\end{subequations}
For later purposes, we show the following lemma which may be used for an alternative definition of $\tilde s$ and $\tilde z$. It also shows the uniqueness of $\tilde s$ and $\tilde z$ for given $s$ and $z$. 
\begin{lemma}
\label{l:zssz}
Let $\tilde s_1 \in \intcc{1;\nu}^\nu$ satisfy \ref{e:tildes:a} in place of $\tilde s$, let $\tilde z_1 \in \intcc{1;\nu}^{\nu}$ satisfy  \ref{e:tildez:a} in place of $\tilde z$. Then
\begin{subequations}
\label{e:zssz}
\begin{align}
\tilde z_1(\tilde s_1(k)) &= k, \text{ and} \\
\label{e:zssz:b}
\tilde s_1(\tilde z_1(k)) &=k 
\end{align}
\end{subequations}
for any $k \in \intcc{1;\nu}$ if and only if $\tilde s_1$ and $\tilde z_1$ satisfy \ref{e:tildes:b} and \ref{e:tildez:b} in place of $\tilde s$ and $\tilde z$, respectively.
\end{lemma}
\begin{proof}
We show \ref{e:tildes:b} for the array $\tilde{s}_1$. The proof of \ref{e:tildez:b} for $\tilde{z}_1$ is similar. We first remark that if $l \in \intcc{1;\nu}$ is an index of both the columns $j$ and $j_0$ then $j = j_0$. This follows immediately from Definition \ref{d:index} and the definition of $i_s$. \\
Let $l$ be an index of the column $j$, hence $\tilde{z}_1(\tilde{s}_1(l))$ is an index of the column $j$. By \ref{e:tildez:a}, $\tilde{z}_1(\tilde{s}_1(l))$ is an index of the column $z(\tilde{s}_1(l))$. On account of the above remark, we have $z(\tilde s_1 (l)) = j$. \\
Conversely, let $k \in \intcc{1;\nu}$ be an index of the column $j$. By \ref{e:tildes:a} and \ref{e:tildez:b} we have
\begin{equation}
\label{e:zssz:1}
s(\tilde{z}_1(\tilde{s}_1(k))) = s(k). 
\end{equation}
Since both $\tilde{z}_1(\tilde{s}_1(k))$ and $k$ are indices of the column $j$ by \ref{e:tildes:b} and \ref{e:tildez:a}, it follows from \ref{e:zssz:1} and the definition of $s$ that $\tilde{z}_1(\tilde{s}_1(k))=k$. The proof of \ref{e:zssz:b} is similar.
\end{proof}
\subsubsection{Data structures for sets}
Realizing the method given in \ref{fig:alg:1} requires the computation of the set $T$ in lines \ref{alg1:updateT:0} and \ref{alg1:updateT:1}. Computing $T$ will require accessing the sets $V \cap \operatorname{NZR}(v)$ for all $v \in \intcc{1;n+r}$. 
For $L=1$, we additionally need to access the set $$ T_0 \defas \{v \in V \ | \  V \cap \operatorname{NZR}(v) = \emptyset \}$$ since the test $V \subseteq \operatorname{NZC}(V)$ in line \ref{alg1:T0empty} is equivalent to the test $T_0 = \emptyset$. 

Therefore, we introduce below appropriate data structures to store the sets $V\cap \operatorname{NZR}(\cdot)$, $T$, $T_0$ and $V$. 

To represent the set $V \cap \operatorname{NZR}(v)$ for any $v \in \intcc{1;n+r}$, we first note that $$\operatorname{NZR}(v) = \{s(k) \ | \ k \in \intco{i_s(v);i_s(v+1)} \ \}.$$ Therefore, we take advantage of the non-uniqueness of $s$ as a part of the CCS-format of $\sX$. In particular, we introduce an integer array $c$ of length $n+r$, and add the following property to the definition of $s$:
\begin{equation}
\label{e:c}
V \cap \operatorname{NZR}(v) = \{ s(k) \ | \ k \in \intco{i_s(v) ; i_s(v) + c(v)}  \ \}.
\end{equation}
In other words, $c(v)$ equals the value $|V \cap \operatorname{NZR}(v)|$ and indicates the last position in $s$ of an element of $V \cap \operatorname{NZR}(v)$. We remark that if $V = \intcc{1;n}$ then $c(v)$ equals the number of $\nz$-entries in the $v$-th column of $\sX$. 

The pair $(s,c)$ stores all information about the sets $V\cap \operatorname{NZR}(\cdot)$. Moreover, the procedure to remove an element $w$ from the set $V \cap \operatorname{NZR}(v)$ can be easily performed: The entry $w$ and the entry in position $i_s(v)+c(v)-1$ are swapped in $s$, and $c(v)$ is decremented by $1$.

The data structures for $T$ and $T_0$ are such that adding, removing and picking of elements can be achieved in constant time. These requirements can be realized using two integer arrays (two for each set) with appropriate functionality. The set $V$ is implemented as a boolean array of length $n+r$ where a '$1$' in the $k$-entry indicates that $k \in V$. 
\subsection{Algorithm}

\label{ss:operations}
An algorithm for the method in \ref{fig:alg:1} is given in \ref{fig:implementation}. Before we prove Theorem \ref{th:main}, we focus on the two operations in \ref{fig:implementation} that determine the complexity of the algorithm: The initialization of the data structures and the execution of line \ref{fig:implementation:remove} which is part of the computation of the sets $T$ and $T_0$. 
\subsubsection{Initialization of the data structures}
\label{sss:init}
The arrays $z$ and $i_z$ can be obtained in time linear in $n+r+\nu$ by transposing $\sX$ when $\sX$ is available in the CCS-format \cite[Section 2]{Gustavson78}. Roughly speaking, the main part of the transposing algorithm in \cite{Saad94} consists of a loop over all indices $k \in \intcc{1;\nu}$. In the body of the loop, $z$ is computed as follows. If $k$ is an index of the column $j$ one sets $z(l)\defas j$ for a suitable index $l$ of the row $s(k)$. Additionally, one may initialize $\tilde s$ by setting $\tilde s(k) \defas l$ in the same loop. Thus, $\tilde s$ clearly satisfies \ref{e:defoftildes}. Similarly, $\tilde{z}$ is then obtained by transposing $\sX^T$. Hence, the arrays $z$, $i_z$, $\tilde z$, $\tilde s$ are initialized with complexity $\mathcal{O}(n+r+\nu)$. 

\begin{figure}[t]
\ifthenelse{\boolean{Forreview}}{\normalsize}{\small}
\algorithmicindent.6em%
\begin{algorithmic}[1]
\item[\textbf{Input}] $s,\tilde s,i_s,z,\tilde z,i_z,c,w,j$
\REQUIRE $\sX_{w,j} = \nz$ and $l \in \intco{i_z(w);i_z(w+1)}$ such that $z(l) = j$.
\STATE $\tilde j \defas \tilde z(l)$ \hspace{3.2cm} (Note that $s(\tilde j) = w$)
\STATE $k \defas \tilde s( i_s(j) + c(j)-1)$
\IF {$c(j)> 1$}  
\STATE 
\label{fig:alg:removefrom:T:3}
Swap entries at positions $\tilde j$ and $i_s(j) + c(j)-1$ in each of the arrays $s$ and $\tilde s$. 
\STATE 
\label{fig:correct:almostdone}
$\tilde z(l) \defas i_s(j) + c(j)-1$
\STATE
$\tilde z(k) \defas \tilde j$
\label{fig:correct:done}
\ENDIF 
\IF {$c(j)>0$}
\STATE
Decrement $c(j)$.
\ENDIF
\item[\textbf{Output}] $s,\tilde s,\tilde z,c$
\end{algorithmic}
\caption{\label{fig:alg:removefrom:T} Procedure to remove $w$ from the representation of $V \cap \operatorname{NZR}(j)$}
 \vspace*{-\baselineskip}
\end{figure}

\begin{figure}[t]
\ifthenelse{\boolean{Forreview}}{\normalsize}{\small}
\algorithmicindent.6em%
\begin{algorithmic}[1]
\item[\textbf{Input}] $s,i_s$
\REQUIRE $L\in \{0,1\}$
\STATE Transpose $\sX$, initialize $c$ and set $V\defas \intcc{1;n}$.
\label{fig:implementation:line:1}
\IF {$L=0$}
\STATE
$T \defas \{ v \in \intcc{1;n+r} \ | \ c(v) = 1\} $
\ELSE
\STATE
$T \defas \{ v \in \intcc{n+1;n+r} \ | \ c(v) = 1\} $
\STATE
$T_0 \defas \{ v \in \intcc{1;n} \ | \ c(v) = 0\} $
\ENDIF
\label{fig:implementation:line:7}
\WHILE {$V \neq \emptyset$}
\label{fig:implementation:line:2}
\IF {$L=0$ \OR $T_0 = \emptyset $}
\label{fig:implementation:line:3}
\IF {$T = \emptyset$}
\STATE \textbf{break}
\ENDIF
\STATE Pick $v \in T$.
\label{fig:implementation:pick}
\STATE $\{w\} \defas V \cap \operatorname{NZR}(v)$ \quad (Note that $w$ is unique.)
\ELSE 
\STATE Pick $w \in T_0$.
\ENDIF
\label{fig:implementation:line:11}
\FORALL{$k \in \intco{i_z(w);i_z(w+1)}$}
\label{fig:implementation:forall}
\STATE $j \defas z(k)$
\STATE Execute the algorithm in \ref{fig:alg:removefrom:T} for $w$ and $j$.
\label{fig:implementation:remove}
\IF {$c(j) = 0$}
\STATE 
$T \defas T \setminus \{j\}$
\label{fig:implementation:removefromT}
\IF {$L=1$ \AND $j \in V$}
\label{fig:implementation:line:17}
\STATE 
$T_0 \defas T_0 \cup \{j \}$
\ENDIF
\ELSIF {$c(j) = 1$}
\IF {$L=0$ \OR $j \notin V$}
\label{fig:implementation:line:21}
\STATE 
$T \defas T \cup \{j \}$
\label{fig:implementation:inserttoT}
\ENDIF
\ENDIF 
\ENDFOR
\IF{$L=1$}
\label{fig:implementation:line:26}
\IF {$c(w) = 1$}
\STATE
$T \defas T \cup \{w \}$
\ELSIF {$c(w) = 0$}
\STATE
$T_0 \defas T_0 \setminus \{ w \}$
\ENDIF
\label{fig:implementation:line:31}
\ENDIF
\label{fig:implementation:line:32}
\STATE 
$V \defas V \setminus \{ w\}$.
\label{fig:implementation:line:newV}
\ENDWHILE
\item[\textbf{Output}] $V$
\end{algorithmic}
\caption{\label{fig:implementation}Algorithm for the method given in \ref{fig:alg:1}}
\vspace*{-\baselineskip}
\end{figure}

It is not hard to see that the initialization of the array $c$ and the data structures for $V$, $T$ and $T_0$ has time complexity $\mathcal{O}(n+r)$. 
\myvspace
\subsubsection{Computation of $T$ and $T_0$}
The computation of $T$ and $T_0$ in lines $3$--$7$ in \ref{fig:alg:1} is implemented by iteratively updating the representation of $T$ and $T_0$. The update consists of two operations:
\begin{compactitem}[$\cdot$]
\item  Removing a row index $w$ (due to line \ref{alg1:removefromV} in \ref{fig:alg:1}) from the representation of the sets $V \cap \operatorname{NZR}(j)$ for all $j$, and 
\item inserting or removing $j$ from $T$ ($T_0$, respectively) depending on the new value of $c(j) = |V \cap \operatorname{NZR}(j)|$. 
\end{compactitem}
The algorithm for the first operation is given in \ref{fig:alg:removefrom:T}: The array $s$ is updated in line \ref{fig:alg:removefrom:T:3} to satisfy \ref{e:c}. Consequently, an update of $\tilde s$ and $\tilde z$ is necessary (lines \ref{fig:alg:removefrom:T:3}, \ref{fig:correct:almostdone} and \ref{fig:correct:done}). The correctness of the updates in \ref{fig:alg:removefrom:T} is formalized in Lemma \ref{l:core} below. Prior to that, we illustrate the iterative computation of $T$ and the use of $z$, $i_z$, $\tilde{s}$ and $\tilde{z}$ by an example. 

\begin{figure*}[t]
\centering
\psfrag{null}[r][t]{\small $0.0$}
\psfrag{0.}[rb][]{\small$0.$}
\psfrag{0.5}[r][]{\small$0.5$}
\psfrag{1.}[r][]{\small$1.0$}
\psfrag{1.5}[r][]{\small$1.5$}
\psfrag{500}[t][r]{\small$500$}
\psfrag{750}[t][r]{}
\psfrag{1000}[t][]{\small$1000$}
\psfrag{1250}[t][r]{}
\psfrag{1500}[t][]{\small$1500$}
\psfrag{1750}[t][r]{}
\psfrag{2000}[t][]{\small$2000$}
\psfrag{2250}[t][r]{}
\psfrag{2500}[t][]{\small$2500$}
\psfrag{title}[][]{\small$n=1000$, $r=250$}
\psfrag{title2}[][]{\small$r=500$, $\nu = 50000$}
\psfrag{y}[][]{\small cpu in ms}
\psfrag{y2}[][]{\small cpu in ms}
\psfrag{x}[l][]{\small $\nu$}
\psfrag{x2}[l][]{\small $n$}
\psfrag{a1}[t][]{\small$1\!\cdot\!10^4$}
\psfrag{a2}[t][]{}
\psfrag{a3}[t][]{\small$3\!\cdot\!10^4$}
\psfrag{a4}[t][]{}
\psfrag{a5}[t][]{\small$5\!\cdot\!10^4$}
\psfrag{a6}[t][]{\small$6\!\cdot\!10^4$}
\psfrag{a6}[t][]{}
\psfrag{a7}[t][]{\small$7\!\cdot\!10^4$}
\includegraphics[width=0.99\textwidth]{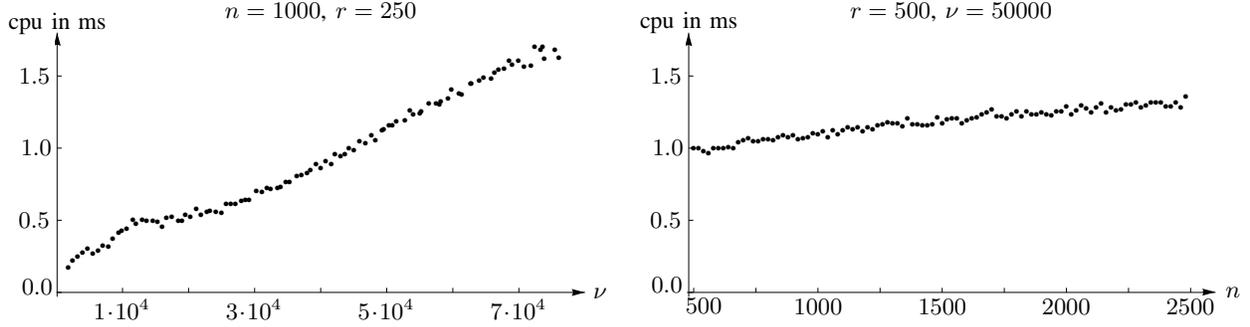} 
\caption{\label{fig:linearity} Run time to verify strong structural controllability for $\lambda = 0$ in dependence of $\nu$ and $n$ for randomly chosen pairs of structural matrices $(\sA,\sB) \in \{0,\nz\}^{n \times (n+r)}$ such that $(\sA,\sB)$ is strong structurally controllable for $\lambda = 0$. $\nu$ denotes the number of $\nz$-entries in $(\sA,\sB)$. The run time to test strong structural controllability for $\lambda \neq 0$ is within a factor $1.1$ of the run time for $\lambda=0$. The underlying implementation of the algorithm in \ref{fig:implementation} was executed on a Intel Core CPU i7-3770S (3.10~GHz).}
\end{figure*}

\begin{example}
\label{ex:tildes}
Let us perform the first steps in the overall algorithm in \ref{fig:implementation} for the pair $(\sA,\sB)$ given by \ref{e:ex:ccs} and for $L=0$.
\\At line \ref{fig:implementation:line:2}, the arrays $s,z,i_s,i_z$ are given as in \ref{t:1}. The array $c$ and the sets $V$ and $T$ are given as follows:
\begin{align*}
c & = (2,1,0,2,0,1,2,1), \\ 
V & = \{1,2,3,4,5,6\}, \\
T & = \{2,6,8\}.
\end{align*}
Since $T \neq \emptyset$ we may pick $2$ from $T$ in line \ref{fig:implementation:pick} to obtain $w = s(i_s(2))=2$. Therefore, for $k \in \intco{i_z(2),i_z(3)}=\{2,3\}$ we have to remove $w=2$ from the sets $V \cap \operatorname{NZR}(z(k))$ as required in line \ref{fig:implementation:remove}. Note that $z(k) = 2,7$ for $k=2,3$, i.e., the $\nz$-entries in row $2$ are precisely in columns $2$ and $7$.

Updating $V \cap \operatorname{NZR}(2)$ and $V \cap \operatorname{NZR}(7)$, respectively, in \emph{constant} time requires the array $\tilde z$ for the following reason. We need to access in constant time those positions $\tilde{j}$ for which $s(\tilde{j}) = w=2$. These indices are required since we need to swap the entries at positions $\tilde{j}$ in $s$ to possibly different positions due to the required property \ref{e:c} of $s$. In order to avoid a search operation, $\tilde z$ is introduced. $\tilde{z}$ provides the required positions $\tilde{j}=3$ and $\tilde{j}=7$ as follows: $\tilde{j}=\tilde z(k) = 3,7$ for $k=2,3$, so $s(3) = s(7) = 2$. 

Since the entry for column $2$ in the array $c$ equals $1$, i.e. $c(2) = 1$, no swap is necessary for updating $V \cap \operatorname{NZR}(2)$. The update is finished by decrementing $c(2)$ to $0$. Therefore, in line \ref{fig:implementation:removefromT}, we remove $2$ from $T$ to temporally obtain $T=\{6,8\}$. In the representation of $V \cap \operatorname{NZR}(7)$, a swap is necessary as $c(7) = 2$. So for $\tilde{j}=7$ (note that we identified $\tilde{j}$ previously), positions $\tilde{j}$ and $i_s(7)+c(7)-1=8$ in $s$ are swapped, and $c(7)$ is decremented to equal $1$. Hence, $7$ is inserted to $T$ in line \ref{fig:implementation:inserttoT}. Thus, we have $T=\{6,7,8\}$ and $V = \{1,3,4,5,6\}$ at line \ref{fig:implementation:line:newV}.

As $s$ has changed, $\tilde z$ and $\tilde{s}$ need to be updated accordingly by means of swapping entries. The access to the required positions is indicated in \ref{t:1} and is similar to the access operations as detailed in this example. The updated arrays are as follows ($z$ remains unchanged, the underlined entries below are those that have changed):

\ifthenelse{\boolean{Forreview}}{\normalsize}{\small}
\begin{align*}
c & = (2,\underline{0},0,2,0,1,\underline{1},1), \\ 
s &= (3,5,2,1,6,4,\underline{3},\underline{2},6), \\
\tilde{s} &= (4,7,2,1,8,6,\underline{5},\underline{3},9), \\
\tilde{z} &= (4,3,\underline{8},1,\underline{7},6,2,5,9). 
\end{align*}
\normalsize
\end{example}

\begin{lemma}
\label{l:core}
Let $\sX_{w,j} = \nz$ and let $l$ be an index of the row $w$ such that $z(l) = j$. Suppose that all inputs in \ref{fig:alg:removefrom:T} satisfy their defining properties for $V \subseteq \intcc{1;n}$. After the termination of the algorithm, $s$ and $c$ satisfy \ref{e:c} for $V\setminus \{w\}$ in place of $V$. Moreover, the arrays $\tilde s$ and $\tilde z$ satisfy \ref{e:defoftildes} and \ref{e:defoftildez}, respectively. 
\end{lemma}
\begin{proof}
The only assertion that requires a proof is the correct update of $\tilde z$ after having changed $s$ in line \ref{fig:alg:removefrom:T:3}. Denote by $\tilde s_0$ and $\tilde z_0$ the arrays $\tilde s$ and $\tilde z$ prior to the execution of the algorithm. Inductively, we may assume that $\tilde s_0$ and $\tilde z_0$ satisfy \ref{e:defoftildes} and \ref{e:defoftildez} in place of $\tilde{s}$ and $\tilde{z}$, respectively. By Lemma \ref{l:zssz} we need to verify \ref{e:zssz} for those positions $k$ of $\tilde z$ and $\tilde s$ whose entries have changed after line \ref{fig:correct:done}. Indeed, \ref{e:zssz} is valid as

\small
\begin{align*}
\tilde s ( \tilde z (l ) ) & = 
\tilde s_0 ( \tilde j) = \tilde s_0(\tilde z_0(l)) = l, \\
\tilde s ( \tilde z(k)) & = \tilde s ( \tilde j) = \tilde s_0 ( i_s(j) + c(j)-1) =k,\\
\tilde z ( \tilde s( \tilde j)) & = \tilde z(\tilde s_0(i_s(j) + c(j)-1)) = \tilde z(k) = \tilde j, \\
\tilde z (  \tilde s(i_s(j) + c(j)-1))& = \tilde z(\tilde s_0(\tilde j)) = \tilde z(l) = i_s(j) + c(j)-1,
\end{align*}
\normalsize
hence the proof is finished.
\end{proof}

Due to the arrays $\tilde s$ and $\tilde z$, the representation of the sets $V\cap \operatorname{NZR}(\cdot)$ can be updated in constant time, which is the key ingredient to Theorem \ref{th:main} as we will see in the following subsection. 

\subsubsection{Correctness and complexity}
Lines \ref{alg1:T0empty}--\ref{alg1:line:16} of \ref{fig:alg:1} clearly correspond to lines \ref{fig:implementation:line:3}--\ref{fig:implementation:line:11} in the algorithm in \ref{fig:implementation}. The computation of $T$ in lines \ref{alg1:updateT:0} and \ref{alg1:updateT:1} in \ref{fig:alg:1} is realized in lines \ref{fig:implementation:forall}--\ref{fig:implementation:line:32} in \ref{fig:implementation} as discussed above. We remark that lines \ref{fig:implementation:line:26}--\ref{fig:implementation:line:32} in \ref{fig:implementation} are required as the set $V$ contains $w$ during the execution of lines \ref{fig:implementation:line:17} and \ref{fig:implementation:line:21} in contrast to line \ref{alg1:updateT:1} in \ref{fig:alg:1}. 
It follows that the outputs of the algorithms in \ref{fig:alg:1} and \ref{fig:implementation} coincide.

The complexity of the algorithm in \ref{fig:implementation} proves Theorem \ref{th:main}. Indeed, lines \ref{fig:implementation:line:1}--\ref{fig:implementation:line:7} are executed in time linear in $n+r+\nu$ (see Subsection \ref{sss:init}). The while loop in line \ref{fig:implementation:line:2} terminates after at most $n$ iterations (see \ref{fig:alg:1}). Together with the for loop in line \ref{fig:implementation:forall}, it yields an overall complexity of $\mathcal{O}(n+\nu)$ for the while loop. Finally, the output of $V$ requires time linear in $n$, so that Theorem \ref{th:main} is proved.

\section{Computational results}
\label{s:results}

In this section, we analyze the performance of the implementation of the algorithm in \ref{fig:implementation} in two aspects: Linearity and the application of the algorithm to the minimization problem as discussed in Section \ref{s:intro}. The programming language in which the algorithm is implemented is C including Fortran routines of \cite{Saad94}. The computations were executed on a Intel Core CPU i7-3770S (3.10~GHz).
\subsection{Linearity}
The linearity in $\nu$ and $n$ of the algorithm in \ref{fig:implementation} for $L=0$ is illustrated in \ref{fig:linearity}. The matrices for which the computational time was recorded are chosen at random such that each is strong structurally controllable for $\lambda = 0$. The run time for $L=1$ and matrices that are strong structurally controllable for every $\lambda \neq 0$ is within a factor $1.1$ of the run time for $L=0$ but it is not illustrated in \ref{fig:linearity}.
\subsection{Minimization}
The matrix $\sB$ given in Example \ref{ex:ccs} is indeed one with the minimum number of columns such that the pair $(\sA,\sB)$ as defined in \ref{e:ex:ccs} is strong structurally controllable. This is verified in $0.56$ milliseconds of cpu time by testing all possible candidates $\sB$. We emphasize that the number of columns required in this case is strictly less than $3$. In contrast, the minimum number of columns obtained by the minimization algorithm presented in \cite{PequitoPopliKarIlicAguiar13} is $3$ due to restricting $\sB$ to have precisely one $\nz$-entry per column. As detailed in the introduction, real applications benefit from a reduced number of columns required for $\sB$ such that $(\sA,\sB)$ is strong structurally controllable.

The investigation of structural properties of electrical networks is quite popular, e.g. \cite{i96iscas,i99ijcta,i03diagnosis,PequitoPopliKarIlicAguiar13}. In \cite[Section IV, p. 418]{PequitoPopliKarIlicAguiar13}, a 5-bus power system is given in terms of a structural matrix $\sA \in \{0,\nz\}^{16\times 16}$. An application of our algorithm in \ref{fig:implementation} to find a structural matrix $\sB$ with a minimum number of columns such that $(\sA,\sB)$ is strong structurally controllable results in a structural matrix $\sB$ having $3$ columns. It takes $437$ seconds for this task using a parallel computation on $6$ threads.

\bibliographystyle{IEEEtran}
\bibliography{IEEEtranBSTCTL,preambles,mrabbrev,strings,fremde,eigeneJOURNALS,eigeneCONF,fremde2}
\end{document}